\documentclass[11pt]{amsart}

\usepackage{amssymb,hyperref,enumitem}
\usepackage{amsmath,amsthm}
\usepackage{color,comment}
\usepackage[normalem]{ulem}

\numberwithin{equation}{section}
\newtheorem{thrm}{Theorem}[section]

\newtheorem{lem}[thrm]{Lemma}
\newtheorem{cor}[thrm]{Corollary}
\newtheorem{prop}[thrm]{Proposition}

\newtheorem{example}[thrm]{Example}

\theoremstyle{definition}
\newtheorem{defin}[thrm]{Definition}

\newtheorem{remark}[thrm]{Remark}

\newcounter{case}
\newenvironment{case}{\refstepcounter{case}
	\medskip \noindent {\bf Case \textup{\thecase.  }}}{\upshape}
\renewcommand{\thecase}{\arabic{case}}

\newcounter{subcase}

\renewcommand{\thesubcase}{\alph{subcase}}

\def\AGL{{\rm AGL}}
\def\PSU{{\rm PSU}}
\def\PSL{{\rm PSL}}
\def\PGL{{\rm PGL}}
\def\GL{{\rm GL}}
\def\Ker{{\rm Ker}}

\def\fix{{\rm fix}}
\def\Aut{{\rm Aut}}

\def\la{\langle}
\def\ra{\rangle}
\def\Z{{\mathbb Z}}
\def\soc{{\mathsf{Soc}}}

\def\supp{{\rm supp}}
\newcommand{\sg}[1]{\langle{#1}\rangle}
\newcommand{\inn}{\mathrm{Inn}}

\newcommand{\strike}[1]{\setbox0\hbox{\hskip1pt \ignorespaces#1\hskip1pt}\hbox to \wd0{\copy0\hss\vrule width\wd0 height2.5pt depth-1.75pt}}

\title[CI-groups for ternary structures]{More CI-groups with respect to ternary relational structures}

\author{Ted Dobson}
\address{
	FAMNIT and IAM, University of Primorska\\
	Muzejska trg 2\\
	Koper 6000, Slovenia}
    \email{ted.dobson@upr.si}

\author{Joy Morris}
\address{Department of  Math and Computer Science\\
	University of Lethbridge\\
	Lethbridge, AB T1K 3M4, Canada}
\email{joy.morris@uleth.ca}

\author{Mikhail Muzychuk}
\address{Department of  Mathematics\\
	Ben-Gurion University of the Negev\\
	Beer-Sheva, Israel}
\email{muzychuk@bgu.ac.il}

\author{Pablo Spiga}
\address{Dipartimento di Matematica e Applicazioni\\
	University of Milano-Bicocca\\
	Via Cozzi 55\\
	20125 Milano, Italy
}
\email{pablo.spiga@unimib.it}
\begin{document}
	
	

	\begin{abstract}
		We explicitly determine all CI-groups with respect to ternary relational structures that have the form $C \times D$, where $C$ is cyclic and $D$ is either a dicyclic group whose order is not divisible by $3$ or a dihedral group.  Such groups are also CI-groups with respect to graphs and digraphs.
	\end{abstract}
	
	\maketitle
	
	\section{Introduction}
	The modern Cayley isomorphism problem began in 1967 when \'Ad\'am \cite{Adam1967} conjectured that any two circulant graphs of order $n$ are isomorphic via a group automorphism of the cyclic group of order $n$. The same problem had already been considered by Bays and Lambossy \cite{Bays1931,Lambossy1931} in the context of designs. 
	Indeed, the notion of a Cayley graph extends naturally to other types of combinatorial objects by calling such an object a {\bf Cayley object} if its automorphism group contains the regular representation of $G$ on its underlying set. The corresponding question becomes:
	\begin{center}
		{\it For a fixed class of combinatorial objects $\mathcal{C}$, determine which groups have the property that any two Cayley objects on $G$ are isomorphic if and only if they are isomorphic by a group automorphism of $G$.}
	\end{center}
	A group with this property is called a {\bf CI-group with respect to the class $\mathcal{C}$}.
	
	From a model-theoretic perspective, the most natural combinatorial objects are relational structures. Let $k$ be a positive integer. A {\bf $k$-ary relation} on a set $X$ is a subset of $X^k$, and a {\bf $k$-ary relational structure} on $X$ is a collection of such relations. Since every $k$-ary relational structure gives rise to a $(k+1)$-ary one, any CI-group with respect to $(k+1)$-ary relational structures is also a CI-group with respect to $k$-ary structures. Thus, this establishes a hierarchy among groups. Note also that when $k=2$, the $2$-ary relational structures are simply coloured digraphs. 
	
	In \cite{Palfy1987}, P\'alfy completely characterised the CI-groups with respect to 4-ary relational structures, and showed moreover that the same groups are CI-groups with respect to all $k$-ary relational structures for $4 \le k \le n$.
	
	\begin{thrm}[P\'{a}lfy, \cite{Palfy1987}]
		A group is a CI-group with respect to $4$-ary relational structures if and only if it is a CI-group with respect to $k$-ary relational structures for every $4 \le k \le n$.
		
		Furthermore, the groups with this property are precisely the cyclic groups $\Z_n$ where $n=4$ or $\gcd(n,\varphi(n))=1$ (where $\varphi$ is Euler’s totient function).
	\end{thrm}
	
	The condition $\gcd(n,\varphi(n))=1$ is exactly the condition ensuring that the cyclic group is the unique group of order $n$.
	
	In this paper we focus on the Cayley isomorphism problem for $3$-ary (that is, ternary) relational structures. The cyclic groups appearing in P\'alfy’s theorem are all CI-groups with respect to ternary relational structures. 
	Since our attention will be on the nonabelian candidates, we record the relevant consequences of the main result of~\cite{Dobson2003} as improved in \cite{DobsonS2013} explicitly.
	\begin{cor}\label{previous-nonabelian}
		If $G$ is a nonabelian CI-group with respect to ternary relational structures, then:
		\begin{itemize}
			\item all nontrivial odd-order Sylow subgroups of $G$ have prime order;
			\item the Sylow $2$-subgroups of $G$ are isomorphic to $\Z_2$, $\Z_4$, or $Q_8$;
			\item $G = U \times V$, where:
			\begin{itemize}
				\item $\gcd(|U|,|V|)=1$;
				\item $U$ is cyclic of odd order $n$ with $\gcd(n,\varphi(n))=1$; and
				\item for some odd $m$ with $\gcd(nm,\varphi(nm))=1$, the group $V$ is either:
				\begin{itemize}
					\item $Q_8$ (with $m=1$),
					\item dihedral of order $2m$, or
					\item dicyclic of order $4m$.
				\end{itemize}
			\end{itemize}
		\end{itemize}
		
		Moreover, if $V=Q_8$ or $V$ is dicyclic and $p\mid n$ is prime, then $p\not\equiv 1\pmod 4$.
	\end{cor}
	
	Prior to this paper, the only known nonabelian CI-groups with respect to ternary relational structures were the dihedral groups of order $2p$ with $p$ prime, as shown by Babai in \cite{Babai1977}.
	
	Our first main result (Corollary~\ref{dihedral cor}) shows that whenever $V$ is dihedral and the conditions of Corollary~\ref{previous-nonabelian} are satisfied, the group $G$ is a CI-group with respect to ternary relational structures. Since this also implies that such groups are CI-groups with respect to binary relational structures, this strengthens the known results for this family with respect to coloured digraphs. Previously, only the dihedral groups in this family, together with direct products of order less than $32$, were known to have the CI property.
	
	Our second main result (Corollary~\ref{dicyclic cor}) shows that if $3\nmid m$, then whenever $V$ is dicyclic and the conditions of Corollary~\ref{previous-nonabelian} are satisfied, the group $G$ is a CI-group if and only if all prime divisors of $m$ are congruent to $3\pmod 4$ (note that Corollary~\ref{previous-nonabelian} already imposes this condition on the prime divisors of $n$).
	
	We expect that our main tool (Theorem~\ref{041021a}) will also aid in resolving the remaining cases, but we leave this for future work.
	
	For nonabelian groups, this leaves unresolved the case where $V$ is dicyclic and $3\mid |V|$, and the case where $V=Q_8$.
	
	Finally, some of our results extend previously known results for CI-groups with respect to graphs and digraphs. In particular, several of the nonabelian groups we identify as CI-groups for ternary relational structures—specifically those with central cyclic subgroups of odd order—were not previously known to be CI-groups for digraphs. This implication is now immediate.
	\section{Preliminaries}\label{sec:preliminaries}
	Given a set $X$, we denote by $S_X$ and $A_X$ the \textbf{symmetric} and \textbf{alternating} groups on $X$, respectively. If $X = \{1, \ldots, n\}$, we write $S_n$ and $A_n$ for the corresponding groups.
	
	If a group $G$ acts on a set $X$ and $Y \subseteq X$, then $G_Y$ and $G_{\{Y\}}$ denote the pointwise and setwise stabilizers of $Y$ in $G$, respectively. If $g(Y) = Y$, then $g^Y \in S_Y$ denotes the permutation of $Y$ \textbf{induced} by $g \in G$; that is, $g^Y \colon Y \to Y$ with $g^Y(y) = g(y)$ for all $y \in Y$. We set $G_{\{Y\}}^Y = \{g^Y : g \in G_{\{Y\}}\}$.
	
	\begin{defin}
		Let $X$ be a set, let $G \leq S_X$ be transitive, and let $\emptyset \ne B \subseteq X$. We call $B$ a \textbf{block} of $G$ if, whenever $g \in G$, either $g(B) \cap B = \emptyset$ or $g(B) = B$. If $B = \{x\}$ for some $x \in X$, or $B = X$, then $B$ is a \textbf{trivial block}. Note that if $B$ is a block of $G$, then so is $g(B)$ for every $g \in G$. The set ${\mathcal B} = \{g(B) : g \in G\}$ is a partition of $X$, called a \textbf{block system of $G$}. If $G$ has a nontrivial block, then $G$ is said to be \textbf{imprimitive}; otherwise, $G$ is \textbf{primitive}.
	\end{defin}
	
	\begin{defin}
		The set of orbits of a normal subgroup of a transitive permutation group $G$ is a block system of $G$, and such a block system is called a \textbf{normal block system}.
	\end{defin}
	
	\begin{defin}Let $G \le S_X$ be a transitive permutation group with a block system ${\mathcal B}$.  
		Each $g \in G$ induces a permutation $g^{\mathcal B}$ on ${\mathcal B}$, defined by
		$g^{\mathcal B}(B) = B'$ if and only if $g(B) = B'$.
		Set
		$G^{\mathcal B} = \{ g^{\mathcal B} : g \in G \}$.
		This yields a homomorphism $G \to S_{\mathcal B}$ given by $g \mapsto g^{\mathcal B}$, whose kernel we denote by $\fix_G({\mathcal B})$.  
		Thus $
		\fix_G({\mathcal B}) = \{ g \in G : g(B) = B \text{ for all } B \in {\mathcal B} \}$,
		and $\fix_G({\mathcal B})$ is the setwise stabilizer of each block of ${\mathcal B}$.
		
	\end{defin}
	
	A block system $\mathcal B$ is normal for a group $G$ if and only if  $\fix_G(\mathcal B)$ acts transitively on each block $B \in \mathcal B$.
	
	\begin{defin}
		Let $G \le S_X$ be transitive and have block systems ${\mathcal B}$ and ${\mathcal C}$. If, whenever $B \in {\mathcal B}$, there exists $C \in {\mathcal C}$ such that $B \subseteq C$, we write ${\mathcal B} \preceq {\mathcal C}$. If $B \subset C$, we write ${\mathcal B} \prec {\mathcal C}$. We say that ${\mathcal B}$ is \textbf{minimal} if it is minimal, with respect to the partial order $\preceq$, among all block systems distinct from $\{\{x\} : x \in X\}$.
	\end{defin}

	\begin{defin}\label{mstep}
		Let $n$ be an integer and $n = p_1^{a_1}p_2^{a_2}\cdots p_r^{a_r}$ be the prime-power decomposition of $n$.  Define $\Omega\colon{\mathbb N}\to {\mathbb N}$ by $\Omega(n) = \sum_{i=1}^ra_i$.   A transitive group $G\le S_n$ is {\bf $m$-step imprimitive} if $G$ has a sequence of block systems ${\mathcal  B}_0 \prec{\mathcal  B}_1\prec \cdots \prec{\mathcal  B}_m$.  If, in addition, each ${\mathcal  B}_i$ is normal, we say that $G$ is {\bf normally $m$-step imprimitive}.
		
		When $m=\Omega(n)$, ${\mathcal  B}_0$ consists of singleton sets, ${\mathcal  B}_m = \{1,\ldots,n\}$, and if $B_{i+1}\in{\mathcal  B}_{i+1}$ and $B_i\in{\mathcal  B}_i$, then $\vert B_{i+1}\vert/\vert B_i\vert$ is prime, $0\le i\le m - 1$.
	\end{defin}
	
	We now define the class of groups with which we will be concerned in this paper.
	
	\begin{defin}\label{190921a}
		Let $n$ be an odd square-free integer. Denote by ${\mathcal R}_n$ the set consisting of the following groups:
		\begin{equation}
			\begin{array}{rl}
				(a) & R = \Z_n \times R_2 \text{ where } R_2 \in \{1, \Z_2, \Z_2^2, \Z_2^3, \Z_2^4, \Z_4, Q_8\};\\[3pt]
				(b) & R = \Z_n \rtimes \langle y \rangle \text{ with } o(y) \in \{2,4,8\},\, y \notin \mathsf{Z}(R) \text{ and } y^2 \in \mathsf{Z}(R),
			\end{array}
		\end{equation}
		where $\mathsf{Z}(R)$ denotes the center of $R$.
		Denote by ${\mathcal R}$ the union $\bigcup_n {\mathcal R}_n$, where $n$ runs over all odd square-free integers.
	\end{defin}
	
	In part~$(b)$, the action of $y$ by conjugation on $\Z_n$ is an arbitrary automorphism of order two. Note that $\mathcal R$ is closed under taking subgroups and quotient groups.
	
	The family $\mathcal R$ includes all the putative nonabelian CI-groups with respect to ternary relational structures (according to Corollary~\ref{previous-nonabelian}). It also includes additional groups, since, for instance, the condition $\gcd(n,\varphi(n))=1$ is omitted. Our reason for including these extra groups is that we believe our results are likely to be useful in determining the CI-status with respect to graphs and digraphs for several families of groups whose status is still undetermined, but which are not CI-groups with respect to ternary relational structures.
	
	\begin{prop}\label{190921b} 
		Let $X$ be a set and let $R \in {\mathcal R}$ act regularly on $X$.  
		Let ${\mathcal B}$ be a block system of $R$. Then, for any $B \in {\mathcal B}$,  $R_{\{B\}}^B$ belongs to ${\mathcal R}$. 
		Moreover, the induced group $R^{\mathcal B}$ on ${\mathcal B}$ contains a regular subgroup belonging to ${\mathcal R}$.
	\end{prop}
	
	\begin{proof}
		Since $R$ acts regularly on $X$,  $R_{\{B\}}$ has order $|B|$, and $R_{\{B\}}^B \cong R_{\{B\}}$ as abstract groups. 
		As ${\mathcal R}$ is closed under taking subgroups, we have $R_{\{B\}}^B \in {\mathcal R}$.
		
		If $R_{\{B\}}$ is normal in $R$, then the induced action $R^{\mathcal B}$ is a regular action of the quotient group $R / R_{\{B\}}$. 
		The conclusion follows, since ${\mathcal R}$ is also closed under taking quotient groups.
		
		If $R_{\{B\}}$ is not normal in $R$, then $R \cong \Z_n \rtimes \langle y \rangle$ as in Definition~\ref{190921a}~$(b)$. 
		As every subgroup of $\Z_n \times \langle y^2 \rangle$ is normal in $R$, the subgroup $R_{\{B\}}$ is conjugate to $K \rtimes \langle y \rangle$, 
		where $K$ is the unique subgroup of $\Z_n$ of order $|B| / o(y)$. 
		In this case, $\fix_R({\mathcal B}) = K \times \langle y^2 \rangle$, and hence 
		$
		(\Z_n \times \langle y^2 \rangle )/ \fix_R({\mathcal B})
		$
		acts regularly on ${\mathcal B}$ and is isomorphic to $\Z_{n/|K|} \in {\mathcal R}$.
	\end{proof}
	
	The following lemma from \cite{Babai1977} has been a critical tool in many CI results, and informs our approach to this problem.
	
	\begin{lem}[Lemma 3.2, \cite{Babai1977}]\label{Babai lemma}
		For a Cayley object $X$ of a group $G$ in a class $\mathcal C$ of
		combinatorial objects,the following are equivalent:
		\begin{enumerate}
			\item $X$ is a CI-object of $G$ in $\mathcal C$;
			\item any two regular subgroups of $\Aut(X)$ that are isomorphic to $G$, are conjugate in $\Aut(X)$.
		\end{enumerate} 
	\end{lem}
	
	We conclude this section with some auxiliary results. 
	
	\begin{thrm}[Corollary 28, \cite{Dobson2003}]\label{dicyclic-conditional}
		Let $n,m$ be integers with $\gcd(nm,\varphi(nm))=1$. Then the following are equivalent for every binary relational structure $X$ on $\Z_n \times {\rm Dic}(m)$. Providing that for every prime divisor $p$ of $nm$, $p \equiv 3 \pmod{4}$, they are equivalent for every ternary relational structure $X$ also: 
		\begin{enumerate}
			\item $X$ has the CI property; and
			\item given any two regular subgroups $R_1, R_2 \cong G$ in $\Aut(X)$, there exists $\psi \in \Aut(X)$ such that $\sg{R_1, \psi^{-1}R_2\psi}$ admits a  block system consisting of $4$ blocks of cardinality $nm$.
		\end{enumerate} 
	\end{thrm}

	\begin{thrm}[Corollary 15, \cite{Dobson2003}]\label{dihedral-conditional}
		Let $n,m$ be integers with $\gcd(nm,\varphi(nm))=1$. Let $X$ be a ternary relational structure on $G=\Z_n\times D_{2m}$. Then the following are equivalent:
		\begin{enumerate}
			\item $X$ has the CI property; and
			\item given any two regular subgroups $R_1, R_2 \cong G$ in $\Aut(X)$, there exists $\psi \in \Aut(X)$ such that $\sg{R_1, \psi^{-1}R_2\psi}$ admits a block system consisting of $2$ blocks of cardinality $nm$.
		\end{enumerate} 
	\end{thrm}

	\section{A generalization of a result of P\'alfy}
	
	P\'alfy showed in \cite{Palfy1987} that a group $G$ of order $n$ is a CI-group with respect to every class of combinatorial objects if and only if $n = 4$ or $\gcd(n,\varphi(n)) = 1$.  The tool that P\'alfy used to reduce the primitive case to the imprimitive case was a result which showed that, with the exception of $A_n$ or $S_n$, for a $2$-transitive group of square-free degree $n$ there is a prime divisor $p$ of $n$ for which $G$ has a cyclic Sylow $p$-subgroup.  Our goal in this section is to prove an analogue of P\'alfy's result for the degrees under consideration in this paper.  Our result (Proposition~\ref{090921a}) is not quite as `clean' as P\'alfy's result, as there are a few small exceptions to the general result. 
	
	\begin{prop}\label{170921a} 
		Let $n$ be a positive integer with $n = 2^e m$, where $0 \le e \le 4$ and $m$ is odd and square-free, and let $G \leq S_n$ be a primitive group. If $G$ contains a semiregular cyclic subgroup with two orbits or a regular abelian subgroup, then  $G$ is $2$-transitive unless one of the following holds:
		\begin{itemize}
			\item $n = p$ is prime and $G < \AGL_1(p)$; or
			\item $n = 10$ and $A_5 \leq G \leq S_5$; or
			\item $n = 16$.
		\end{itemize} 
	\end{prop}
	
	\begin{proof}
		Assume first that $n$ has a unique prime divisor, say $p$. 
		If $p$ is odd or $n = p = 2$, then $n = p$, and by Burnside’s Theorem~\cite[Theorem~3.5B]{DixonM1996}, either $G < \AGL_1(p)$ or $G$ is $2$-transitive 
		(the statement of Theorem~3.5B also includes the possibility $G = \AGL_1(p)$, but this group is $2$-transitive). 
		If $n > p = 2$, then $n = 4, 8$, or $16$. 
		Since every primitive group of degree $4$ or $8$ is $2$-transitive, we are left with $n = 16$, which is an exception.
		
		Now assume that $n$ is divisible by two primes, at least one of which is odd—call it $p$. 
		If $G$ contains a regular abelian subgroup, then its Sylow $p$-subgroups are cyclic, and by~\cite[Theorem~25.4]{Wielandt1964}, $G$ is $2$-transitive. 
		Otherwise, by hypothesis, $G$ has a semiregular cyclic subgroup $C$. 
		As $C$ is semiregular, each of its orbits has length $k = n/2$. 
		All primitive groups containing a cyclic subgroup with two orbits were classified in~\cite[Theorem~3.3]{Muller2013}. 
		After eliminating groups listed in~\cite[Theorem~3.3]{Muller2013} whose degree is a prime power or a square 
		(impossible since $p^2 \nmid n$), as well as $2$-transitive groups, 
		the only remaining possibility that contains a cyclic subgroup with two orbits of equal length is 
		$k = 5$ and $G \in \{A_5, S_5\}$.
	\end{proof}
	
	Our next result will rely on Zsigmondy’s Theorem, which we now state.
	
	\begin{thrm}[Zsigmondy’s Theorem]\label{Zsigmondythrm}
		Let $a$ and $k$ be integers with $a,k\ge 2$. Then there exists a prime $p$ such that $p$ divides $a^k - 1$ but does not divide $a^\ell - 1$ for any positive integer $\ell < k$, with the following exceptions:
		\begin{itemize}
			\item $a = 2$ and $k = 6$; or
			\item $a + 1$ is a power of two and $k = 2$.
		\end{itemize}
	\end{thrm}
	
	\begin{defin}
		A prime $p$ satisfying the conclusion of Zsigmondy’s Theorem is called a \textbf{primitive prime divisor}.
	\end{defin}
	If $r$ is a primitive prime divisor of $a^k - 1$, then $a$ is coprime to $r$, and $k$ is the order of $a$ in $\Z_r^*$. 
	Therefore, $r \equiv 1 \pmod{k}$, and in particular, $r - 1 \ge k$.
	
	\begin{defin}
		For a group $G$, we denote by $\soc(G)$ the \textbf{socle} of $G$, that is, the subgroup generated by all minimal normal subgroups of $G$.
	\end{defin}
	
	In our generalisation of P\`alfy’s result, we show that, in all but a small number of cases, there exists a Sylow $r$-subgroup of $G$ of order $r$, for some odd prime $r$ dividing the degree. 
	This implies that the Sylow $r$-subgroups act semiregularly, since  no point stabiliser has order divisible by $r$.
	
	\begin{prop}\label{090921a} 
		Let $n$ be a positive integer with $n = 2^e m$, where $0 \le e \le 4$ and $m$ is odd and square-free, and let $G \leq S_n$ be a primitive subgroup distinct from $A_n$ and $S_n$. 
		If $G$ contains a semiregular cyclic subgroup with two orbits or a regular abelian subgroup, then one of the following holds:
		\begin{enumerate}
			\item\label{pablo1} $n = 8$ and $G = \Z_2^3 \rtimes G_0$, where $G_0 \leq \GL_3(2)$ acts transitively on the non-zero vectors;
			\item\label{pablo2} $n = 16$ and $G = \Z_2^4 \rtimes G_0$, where $G_0 \leq \GL_4(2)$;
			\item\label{pablo22}$n=8$ and $G=\mathrm{PSL}_2(7)$ or $G=\mathrm{PGL}_2(7)$;
			\item\label{pablo3} $n = 12$ and $G = M_{12}$;
			\item\label{pablo4} $n = 24$ and $G = M_{24}$;
			\item\label{pablo5} there exists an odd prime $r \mid n$ such that a Sylow $r$-subgroup of $G$ has order $r$.
		\end{enumerate}
	\end{prop}

	\begin{proof}
		By Proposition~\ref{170921a}, the group $G$ is either $2$-transitive, has degree $10$ or $16$, or satisfies $n = p$ for some prime $p$ with $G < \AGL_1(p)$. 
		If $n = 10$, then $G \in \{A_5, S_5\}$ and satisfies~\eqref{pablo5} with $r = 5$. 
		If $n = 16$, then using the computer algebra system Magma~\cite{MAGMA} we deduce $G$ is affine primitive and satisfies~\eqref{pablo2}. 
		If $n = p$ is prime and $G < \AGL_1(p)$, then~\eqref{pablo5} holds with $r = p$. 
		Hence, for the remainder of the proof, we assume that $G$ is $2$-transitive.
		
		If $G$ is affine $2$-transitive, then $n$ is a power of  some prime $p$. 
		For $p = 2$, excluding $A_n$ and $S_n$ eliminates $e = 1,2$. 
		If $e = 3$ or $4$, then $G$ satisfies~\eqref{pablo1} or~\eqref{pablo2}. 
		If $p$ is odd, then $n = p$, and $G$ satisfies~\eqref{pablo5}.
		
		If $n$ is square-free, then by~\cite[Lemma~2.1]{Palfy1987} there exists a Sylow $p$-subgroup of order $p$ for some $p \mid n$. 
		Although not stated explicitly, P\`alfy’s proof shows that $p \ne 2$, giving~\eqref{pablo5}.
		
		We may therefore assume $e \ge 2$, i.e., $4 \mid n$. 
		Let $T = \soc(G)$. 
		We inspect all non-affine $2$-transitive groups in~\cite[Table~7.4]{Cameron1999} whose degree is divisible by four, aiming to show that $G$ satisfies~\eqref{pablo22},~\eqref{pablo3},~\eqref{pablo4} or~\eqref{pablo5}.
		
		\begin{case} $T\cong \PSL_d(q), n = \frac{q^d-1}{q-1}, [G:T]\mid \gcd(d,q-1)f$ where $q=p^f$ and $p$ is a prime.
		\end{case}
		
		Assume first that $p^{df} - 1$ has a primitive prime divisor $r$. 
		Then $r$ is odd, divides $n = \frac{q^{d} - 1}{q - 1}$, and satisfies $r \equiv 1 \pmod{df}$. 
		Hence $r \nmid |G|/n$, so a Sylow $r$-subgroup of $G$ has order $r$, yielding~\eqref{pablo5}. 
		
		If $p^{df} - 1$ has no primitive prime divisor, then by Zsigmondy’s Theorem  either 
		$df = 2$ and $n = q + 1$ is a power of $2$, or $p = 2$ and $df = 6$.  
		
		In the first case, $df = 2$ and $n = q + 1$ imply $d = 2$, $f = 1$, and $q = p$. 
		Since $n = 2^e m$ with $m$ odd and $e \le 4$, we obtain $e = 3$, $p = 7$, and $n = 8$, corresponding to~\eqref{pablo22}.  
		If instead $df = 6$ and $p = 2$, then $n$ is odd, contradicting the fact that $n$ is divisible by $4$.

		\begin{case} 
			$T\cong {\rm Sp}_{2d}(2), n = 2^{d-1}(2^d \pm 1), d \geq 3$.
		\end{case}
		
		Since $32 \nmid n$, we have $d \le 5$. 
		If $d = 3$ and $n = 28$, then $7 \mid |G|$ but $7^2 \nmid |G|$, giving~\eqref{pablo5} with $r = 7$. 
		Similar arguments apply for $d = 4$, $n = 8 \cdot 17$ with $r = 17$; 
		for $d = 5$, $n = 16 \cdot 31$ with $r = 31$; and for $d = 5$, $n = 16 \cdot 33$ with $r = 11$.  
		
		The remaining case with $d = 3$ and $n = 36$ is excluded, since $n$ is divisible by $9$ and hence does not satisfy our hypothesis on the degree. 
		The final case is $d = 4$ and $n = 8 \cdot 15$. 
		A Magma computation shows that in this action $G$ contains neither an abelian regular subgroup nor a semiregular subgroup with two orbits; hence $G$ does not satisfy our hypothesis.
		
		\begin{case} $T\cong \PSU_3(q), n = q^3+1$.
		\end{case}
		
		In this case, $|G|$ divides 
		$|\Aut(T)| = (q^3 + 1)q^3(q^3 - 1) \cdot \gcd(3, q + 1) \cdot f$,
		where $q = p^f$ and $p$ is prime. 
		Let $r$ be a primitive prime divisor of $p^{6f} - 1 = q^6 - 1$. 
		Then $r \mid q^3 + 1$ and $r \equiv 1 \pmod{6f}$. 
		Hence $r$ is coprime to $q^3(q^3 - 1) \cdot \gcd(3, q + 1) \cdot f$, yielding~\eqref{pablo5}. 
		
		If $p^{6f} - 1 = q^6 - 1$ has no primitive prime divisor, then by Zsigmondy’s Theorem we have $6f = 6$ and $p = 2$. 
		In this case, $n = q^3 + 1$ is odd, contradicting $4 \mid n$.
		
		\begin{case} $T\cong Sz(q), n=q^2+1, q = 2^{2s+1} > 2$.
		\end{case}
		
		In this case $n$ is odd, contradicting $
		4\mid n$.
		
		\begin{case} $T\cong R(q), n=q^3+1, q = 3^{2s+1} > 3$.
		\end{case}
		
		In this case $|R(q)|=q^3(q^3+1)(q-1)$ and  $[G:T]\,|\, (2s+1)$. Let $r$ be a primitive prime divisor of $3^{6(2s+1)} - 1 = q^6 - 1$.   Then $r$ divides $q^3+1=n$ and $r\equiv 1\mod{6(2s+1)}$.  Therefore $r$ is coprime to $[G:T]$, $q$ and $q-1$. We conclude that~\eqref{pablo5} holds.
		
		\begin{case} $G=T\cong M_{11},n=12$.
		\end{case}
		
		Since $|G_\omega|$ is even, $G_\omega$ contains an involution. Since $M_{11}$ has a unique conjugacy class of involutions \cite{Atlas}, this implies that every involution of $M_{11}$ has a fixed point. Therefore, this representation of $M_{11}$ has no regular subgroup nor any semiregular cyclic subgroups of order $6$.
		
		\begin{case} $G=T\cong M_{12},n=12$, or $G=T\cong M_{24},n=24$.
		\end{case}
		
		These situations are~\eqref{pablo3} and~\eqref{pablo4}.

		\begin{case} All other possibilities.
		\end{case}
		
		Since $n=4m$, $m$ square-free, we are left with only three cases:
		$$(T,n)\in \{(\PSL_2(8),28), ({\rm HS}, 176), ({\rm Co}_3,276)\}.$$
		The group $\PSL_(8)$ has order $504$ and hence part~\eqref{pablo5} follows with $r=7$; similarly, ${\rm HS}$ satisfies~\eqref{pablo5} with $r=11$, and ${\rm Co}_3$ satisfies~\eqref{pablo5} with $r=23$.
	\end{proof}

	\section{Some Tools}\label{sec:some tools}

	\begin{defin}
		Let $X$ be a set. For a permutation group $H \le S_X$, the \textbf{support} of $H$, denoted $\supp(H)$, 
		is the set of all $x \in X$ for which there exists $h \in H$ with $h(x) \ne x$.
	\end{defin}
	
	\begin{defin}
		Let $G \le S_X$ and $Y \subseteq X$ be $G$-invariant. 
		We denote by $\mathsf{Orb}(G,Y)$ the set of $G$-orbits on $Y$.
	\end{defin}
	
	The next lemma is a variation of~\cite[Lemma~2.2]{Dobson2009}.

	\begin{lem}\label{support blocks}\label{simplegroupblock-cor}
		Let $G\leq S_X$ be transitive with a normal block system ${\mathcal  B}$. Abbreviate $N=\fix_G({\mathcal  B})$ and fix an arbitrary $B_0\in\mathcal{B}$.  If $T=\soc(N^{B_0})$ is a transitive nonabelian simple group, then
		\begin{enumerate}
			\item\label{parta} $\soc(N)= T_1\times \cdots \times T_k$, where all simple factors are isomorphic to $T$. The sets $C_i:=\supp(T_i), 1\le i\le k$ form a block system ${\mathcal C}$ of $G$ with ${\mathcal  B}\preceq {\mathcal  C}$;
			\item\label{parta1}
			for any $T_i$  and any block $B\in\mathcal{B}$ contained in $C_i$, the point-wise stabilizers $N_{B}$ and $N_{C_i}$ are equal and $N_{C_i}\cap T_i =1$;
			\item\label{partb} suppose that $\soc(N)$ contains two isomorphic semiregular subgroups $U$ and $V$.  If any two semiregular subgroups of $T$ isomorphic to $U^{B_0}$ are conjugate in $T$, then there exists $\delta\in \fix_G({\mathcal  B})$ such that $\mathsf{Orb}(U,X)=\mathsf{Orb}(V^\delta,X)$;
			\item\label{partc} an element $g\in N$ belongs to  $\soc(N)$ if and only if $g^{B}\in \soc(N^{B})$ holds for each $B\in\mathcal{B}$.
		\end{enumerate}
	\end{lem}
    
	\begin{proof}  Since $G$ acts transitively on ${\mathcal  B}$ and $N\unlhd G$, the group $N^{B_0}$ is permutation isomorphic to $N^B$ for each $B\in{\mathcal  B}$ by \cite[Theorem 1.6A (ii)]{DixonM1996}. In particular, $\mathsf{Soc}(N^{B_0})\cong \mathsf{Soc}(N^{B})$ for each $B\in{\mathcal  B}$.
		
		Part~\eqref{parta}. Since for any $B\in\mathcal{B}$ the map $g\mapsto g^B,g\in N$, is a group homomorphism, $\mathsf{Soc}(N)^B\unlhd \mathsf{Soc}(N^B)\cong T$. Therefore  either $\mathsf{Soc}(N)^B$ is trivial or $\mathsf{Soc}(N)^B = \mathsf{Soc}(N^B)$. The first case is impossible. Thus $\mathsf{Soc}(N)^B = \mathsf{Soc}(N^B) = T$.
		
		Every subgroup $L\leq N$ is a subdirect product of its projections $L^{B},B\in{\mathcal  B}$. Therefore $\mathsf{Soc}(N)$ is a subdirect subgroup of $T^m$ where $m=|{\mathcal  B}|$. Since $T$ is a simple group, $\mathsf{Soc}(N)\cong T^k$ for some $k\leq m$. Thus $\mathsf{Soc}(N) = T_1\times \cdots \times T_k$, where $T_1,\ldots,T_k$ are normal subgroups of $\mathsf{Soc}(N)$ isomorphic to $T$. It remains to show that $\supp(T_i)\cap \supp(T_j)=\emptyset$ whenever $i\neq j$.
		Indeed, if $\supp(T_i)\cap \supp(T_j)$ is nonempty, then there exists
		$B\in {\mathcal  B}$ such that $T_i^{B}\neq 1 \neq T_j^{B}$. Then $T^2\cong (T_i\times T_j)^{B}\leq \mathsf{Soc}(N)^{B} \cong T$, a contradiction.
		Thus ${\mathcal  C} = \{\supp(T_i):1\le i\le k\}$ is a block system of $G$ with ${\mathcal  B}\preceq{\mathcal  C}$.
		
		Part~\eqref{parta1}. To ease the proof we'll assume that $i=1$. The inclusion $N_{C_1}\leq N_{B}$ is an immediate consequence of $B\subseteq C_1$. To prove the converse inequality $N_{C_1}\geq N_{B}$ it is sufficient to show that $N_{B}=N_{B'}$ for any block $B'\in\mathcal{B}$ contained in $C_1$.
		
		Notice that both             $N_{B}$ and $N_{B'}$ are normal in $N$. Assume towards a contradiction that  $N_{B}$ and $N_{B'}$ are distinct. Then $(N_{B'})^{B}$ is non-trivial.
		It follows from $N_{B'}\trianglelefteq N$ that 
		$1\neq (N_{B'})^{B}\trianglelefteq N^{B}$. But the socle of $N^{B}$ is a simple group $T_1^{B}\cong T_1$. Therefore $T_1^{B}\leq (N_{B'})^{B}$. Since $T_1$ is nonabelian, we conclude that $[T_1^{B},(N_{B'})^{B}]\neq 1$ implying $[T_1,N_{B'}]\neq 1$. Both $T_1$ and $N_{B'}$ are normal in $N$. Therefore $[T_1,N_{B'}]\leq T_1\cap N_{B'}$ implying $T_1\cap N_{B'}\neq 1$. The minimality of $T_1$ implies that $T_1\leq N_{B'}$ which, in turn, yields $\supp(T_1)\cap B'=\emptyset$, a contradiction.  
		
		Part~\eqref{partb}. For $1\le i\le k$, choose a block  of ${\mathcal  B}$ contained in $C_i\in{\mathcal  C}$, we denote it as $B_i$. Let $g_i\in G$ satisfy $g_i(B_i) = B_i$. Then $g_i(U^{B_i})g_i^{-1}$ and $g_i(V^{B_i})g_i^{-1}$ are semiregular subgroups of $g_i(T_i^{B_i})g_i^{-1} = T$ isomorphic to $U^{B_i}$, and so there exists $t_i\in T$ with $g_i(U^{B_i})g_i^{-1} = t_ig_i(V^{B_i})g_i^{-1}t_i^{-1}$.  Then $U^{B_i} = g_i^{-1}t_ig_i(V^{B_i})g_i^{-1}t_i^{-1}g_i$.  Let $\delta_i\in T_i$ be the unique element such that $\delta_i^{B_i} =  (g_i^{-1}t_ig_i)^{B_i}$.
		Then $U^{B_i} = g_i^{-1}t_ig_i(V^{B_i})g_i^{-1}t_i^{-1}g_i = (\delta_i^{-1} V\delta_i)^{B_i}$. As $C_i = \supp(T_i)$, this implies that $U^{C_i} = (\delta_i^{-1} V\delta_i)^{C_i}$, $1\le i\le k$.  Set $\delta = \delta_1 \cdots \delta_k$. It follows from $\supp(\delta_i)\subseteq C_i$ that $\delta^{C_i}=\delta_i^{C_i}$. Therefore, $U^{C_i} = (\delta^{-1} V\delta)^{C_i}$ implying $\mathsf{Orb}(U,C_i) = \mathsf{Orb}(V^\delta,C_i)$. Now the statement follows from $\mathsf{Orb}(U,X) =\bigcup_{i=1}^k \mathsf{Orb}(U,C_i)=\bigcup_{i=1}^k \mathsf{Orb}(V^\delta,C_i)=\mathsf{Orb}(V^\delta,X) $.
		
		Part~\eqref{partc}. The ``only if'' part follows from the equality $\soc(N)^{B} = \soc(N^{B})$ shown in the proof of part~\eqref{parta}. It remains to prove the ``if'' part. If $n^{B}\in \soc({N^{B}})=T_i^{B}$ and 
		$n^{B'}\in \soc({N^{B'}})=T_i^{B'}$
		 for arbitrary $B,B'\subseteq C_i$, then $nt^{-1}\in N_{B}=N_{B'}\ni n{t'}^{-1}$ for some $t,t'\in T_i$. Therefore, $gt=n=g't'$ for some $g\in N_{B}= N_{B'}\ni g'$. This yields us $t't^{-1}\in N_B = N_{C_i}$ which,  in turn,  implies $t't^{-1}\in T_i\cap N_{C_i}=1$, Therefore $t'=t$ implying that there exists a unique $t_i \in T_i$ with $n^B=t_i^B$ for any block $B\in\mathcal{B}$ contained in $C_i$. Thus we proved that for any $i=1,...,k$ there exists $t_i\in T_i$ with $nt_i^{-1}\in N_{C_i}$. We claim that $n=t_1\cdots t_k\in\soc(N)$.  Since for each pair $i\neq j$ we have $t_i\in N_{C_j}$, we conclude that  
		 $$nt_1^{-1}\cdots t_k^{-1} = nt_i^{-1} t_1^{-1}\cdots t_{i-1}^{-1} t_{i+1}^{-1}\cdots t_k^{-1} \in N_{C_i}. $$
		 Since the above inclusion holds for each $i=1,...,k$, we conclude                 
		$$nt_1^{-1}\cdots t_k^{-1} \in \bigcap_{i=1}^k N_{C_i} =  1.
		$$
		as desired.
		
	\end{proof}
	
	\begin{lem}\label{altbottom1}
		Let $T \le S_k$ be one of the groups $A_k$, $M_{12}$, or $M_{24}$ 
		(in the latter two cases $k = 12$ and $k = 24$, respectively). 
		Then any two semiregular subgroups of odd prime order $p$ contained in $T$ are conjugate in $T$.
	\end{lem}
	
	\begin{proof}
		Any two semiregular subgroups $R_1, R_2 \le T$ of order $p$ are conjugate in $S_k$. 
		If $k = p$, then by Sylow’s theorem they are conjugate in $A_k$. 
		If $p < k$, then each $R_i$ has $k/p > 1$ orbits, and hence 
		${\bf C}_{S_k}(R_i)$ contains odd permutations, implying that $R_1$ and $R_2$ are conjugate by an even permutation. 
		This completes the proof for $T = A_k$.
		
		For $T = M_{12}$ ($k = 12$, $p = 3$), the semiregular subgroups are cyclic of order $3$, 
		generated by permutations of cycle type $3^4$, all conjugate in $M_{12}$ 
		(see~\cite{BuekenhoutR1988} or~\cite{Atlas}).  
		For $T = M_{24}$ ($k = 24$, $p = 3$), the argument is identical: 
		all elements of cycle type $3^8$ are conjugate in $M_{24}$ (see~\cite{Atlas}).
	\end{proof}
	\begin{lem}\label{2pbottom} Let $R$ and  $T$ be isomorphic transitive groups such that $R$ and $T$ have normal semiregular $p$-subgroups $R_p$ and $T_p$, respectively. Suppose that $G=\la R,T\ra$ has a block system ${\mathcal B}$ such that $R_p,T_p\le\fix_{G}({\mathcal B})$ and both $(R_p)^B$ and $(T_p)^B$ have the same orbits as a Sylow $p$-subgroup of $\fix_{G}({\mathcal B})^B$, $B\in{\mathcal B}$. Then there exists $\delta\in G$ such that $\la R, T^\delta\ra$ has a normal block system ${\mathcal C}=\mathsf{Orb}(R_p,X) = \mathsf{Orb}(T^\delta_p,X) $ with blocks of size $|R_p|$. \end{lem}

	\begin{proof}
		Since $R_p$ and $T_p$ are $p$-subgroups of $N = \fix_G({\mathcal  B})$, there exists $\delta\in N$ such that $R_p$ and $(T_p)^\delta$ are contained in a Sylow $p$-subgroup of $N$, say $P$. Let $B\in{\mathcal  B}$. It follows from $R_p\leq P$ that $(R_p)^B\leq P^B$. By assumption $(R_p)^B$ has the same orbits as a Sylow $p$-subgroup of $N^B$. Therefore, $\mathsf{Orb}((R_p)^B,B)=\mathsf{Orb}(P^B,B)$. Analogously,  $\mathsf{Orb}(((T^\delta)_p)^B,B) = \mathsf{Orb}(P^B,B)$, implying that $\mathsf{Orb}(((T^\delta)_p)^B,B) = \mathsf{Orb}((R_p)^B,B)$.  As the latter equality holds for any block $B\in{\mathcal  B}$, we conclude that $\mathsf{Orb}((T^\delta)_p,X) = \mathsf{Orb}(R_p,X)$. Since $R_p\unlhd R$, the partition ${\mathcal  C} = \mathsf{Orb}(R_p,X)$ is a block system of $R$. Analogously, ${\mathcal  C} = \mathsf{Orb}((T^\delta)_p,X)$ is a block system of  $T^\delta$.  Thus, ${\mathcal  C}$ is a block system of $\la R, T^\delta\ra$. It is normal, because every block of ${\mathcal  C}$ is an orbit of $R_p$.
	\end{proof}
	
	\begin{lem}\label{smaller-blocks}
		Let $R\in{\mathcal R}$ satisfy $|R|_2=2^e$ with $e\in\{3,4\}$. Let $R_1$ and $R_2$ act regularly on a set $\Omega$, with $R\cong R_1\cong R_2$. Assume $G:=\sg{R_1,R_2}$ admits a 
		minimal block system $\mathcal{B}$ with blocks of size $2^d$, $d\ge 3$. Moreover, assume that $G_{\{B\}}^B$  has a nonabelian simple socle. Then there is some $g \in G$ such that $\sg{R_1, R_2^g}$ admits a block system with blocks of size $2$.
	\end{lem}
	
    
	\begin{proof}
		Let $B \in \mathcal{B}$. By minimality of $\mathcal{B}$, the group 
		$H = G_{\{B\}}^B$ is a primitive group of degree $2^d$ containing the regular subgroup $(R_1)_{\{B\}}^B$. 
		By examining the primitive groups of degree $2^d$ and by taking in account the given hypothesis, we deduce that $H$ is isomorphic to either $A_{2^d}$, $S_{2^d}$ or $\PGL_2(7)$.
		
		Denote $N=\fix_G(\mathcal{B})$, as before. Let $Q_i\leq R_i,i=1,2$ be subgroups 
		acting regularly on $B$. Then $|Q_i|=2^d\geq 8$. Since $Q_i$ contains a normal subgroup of $R_i$ of index at most 2, 
		we conclude that $|Q_i\cap N|\geq 4$ implying 
		$N^{B'}\geq (Q_i\cap N)^{B'}\cong Q_i\cap N$ for each $B'\in \mathcal{B}$. Therefore $N^B$ is a non-trivial normal subgroup of $H\cong \mathrm{PGL}_2(7),A_{2^d}, S_{2^d}$.
		In particular $[H:N^B]\leq 2$ and $\soc(N^B)=\soc(H) = H\cap A_B$.
		Since the subgroup $N\cap Q_i$ is a semiregular $2$-group, it contains a semiregular subgroup of order $2$, say $I_i=\sg{\iota_i}$. Since permutation $\iota_i^B$ is even, $\iota_i^B\in \soc(N^B)\in\{\mathrm{PSL}_2(7),A_{2^d}\}$ for each  $B\in\mathcal{B}$. 
		By part~\eqref{partc} of Lemma~\ref{support blocks} $\iota_i\in\soc(N),i=1,2$. Therefore, the subgroups $I_i=\sg{\iota_i},i=1,2$ are contained in $\soc(N)$. 
		
		Note also that in any of $A_8$, $A_{16},$ and $\mathrm{PSL}_2(7)$, any two semiregular subgroups of order $2$ are conjugate. 
		Thus, by Lemma~\ref{support blocks}~\eqref{partb}, there exists $\delta \in N$ such that
		$
		\mathsf{Orb}(I_1,X) = \mathsf{Orb}(I_2^\delta,X).
		$
		Since $|I_1| = |I_2^\delta| = 2$, this forces $I_1 = I_2^\delta$. 
		Therefore $R_1$ and $R_2^\delta$ share a common normal subgroup $I_1 = I_2^\delta$, and the orbits of this subgroup form blocks of size $2$ for the action of $\langle R_1, R_2^\delta\rangle$.
	\end{proof}

	\section{The Main Result}\label{sec:main}
	
	We will often need to refer to a Sylow $p$-subgroup $P$ of a group $R$.  
	For convenience, we denote such a subgroup by $R_p$.  
	
	\begin{prop}\label{300921a} Let $R$ and $T$ be isomorphic regular subgroups of $S_{2^en}$ belonging to $\mathcal{R}$ with $0\le e\le 4$ and $n$ an odd square-free integer, 
		let $G = \la R,T\ra$  and let $\mathcal{B}$ be a minimal   block system of $G$ and let $k$ be the cardinality of the elements of $\mathcal{B}$. If $k$ is divisible by an odd prime $p$, then there exists $g\in G$ such that $\sg{R,T^g}$ has a normal block system with block size $p$.
	\end{prop}
	
	\begin{proof}
		We set $N = \fix_G({\mathcal B})$ and let $p$ be an odd prime divisor of $k$. As $R \in \mathcal R$, it is readily seen  that $R_p \triangleleft R$. Since $R$ acts transitively on ${\mathcal  B}$ and $R_p\triangleleft R$, all orbits of $R_p$ on ${\mathcal  B}$ have cardinality $1$ or $p$. As $|{\mathcal  B}|$ is coprime to $p$,  $R_p$ acts trivially on ${\mathcal  B}$, i.e. $R_p\leq N$. Analogously, $T_p\leq  N$. 
		
		By the minimality of $\mathcal B$, the group $G_{\{B\}}^B$, $B\in{\mathcal  B}$,  is a primitive subgroup of $S_{B}$ and contains $(R_{\{B\}})^B$, which is a regular subgroup that by Proposition~\ref{190921b} is in $\mathcal{R}$.  Thus either $R_{\{B\}}^B$ is an abelian regular subgroup, or it contains a cyclic semiregular subgroup having two orbits. Applying Proposition~\ref{090921a} to $G_{\{B\}}^B$, and recalling that $k$ is divisible by an odd prime, we obtain that either there exists an odd prime $p \mid k$ such that $(R_p)^B$ is a Sylow $p$-subgroup of $G^B$, or $
		\soc(G^B) \cong A_k,\; M_{12},\; M_{24}$,
		where in the second case $k = 12$ and $p = 3$, and in the third case $k = 24$ and $p = 3$.
		
		Assume that $(R_p)^B$ is a Sylow $p$-subgroup of $G^B$. Then $(R_p)^B$ and $(T_p)^B$ are Sylow $p$-subgroups of $N^B$. By Lemma \ref{2pbottom}, there exists $g\in G=\sg{R,T}$ such that $\mathsf{Orb}(R_p,X)$ is a normal block system of $\sg{R,T^g}$ with block size $p$.
		
		In the remaining case $\soc(N^B)=\soc(G^B)\cong A_k,M_{12},M_{24}$. Let $p$ be an odd prime divisor of $k$.
		The subgroups $R_p$ and $T_p$ are semiregular and their images $(R_p)^B,(T_p)^B$ are contained in $\soc(N^B)\cong A_k,M_{12},M_{24}$. By Lemma~\ref{altbottom1},  $(R_p)^B,(T_p)^B$ are conjugate  in $\soc(N^B)$. From Lemma~\ref{support blocks}~\eqref{partb}, there exists $\delta\in N$ with $\mathcal{P}=\mathsf{Orb}(R_p,X) = \mathsf{Orb}((T_p)^\delta,X)$. Since $(T_p)^\delta = (T^\delta)_p$,
		the orbit partition $\mathcal {P}$ is $\sg{R,T^\delta}$-invariant. Thus $\mathcal{P}$ is a normal block system of $\langle R,T^\delta\rangle$  with block size $p$.
	\end{proof}
	
	After some definitions, we are ready for the main result of this section.
	
	\begin{defin}
		For a positive integer $n$, we denote by $\pi(n)$ the set of prime divisors of $n$.
	\end{defin}
	
	\begin{defin}
		Let $G \le S_X$ be transitive with block systems ${\mathcal B} \prec {\mathcal C}$.  
		For $C \in {\mathcal C}$, we write 
		$C/{\mathcal B}$
		for the set of blocks of ${\mathcal B}$ contained in $C$.  
		Set $
		{\mathcal C}/{\mathcal B} = \{ C/{\mathcal B} : C \in {\mathcal C} \}$.
		Note that ${\mathcal C}/{\mathcal B}$ is a block system of $G^{{\mathcal B}}$.
	\end{defin}
	
	\begin{thrm}\label{041021a} Let $R \in \mathcal {R}$, with $|R| = 2^e n$ where $0 \le e \le 4$ and $n > 1$ is odd and square-free, be a regular subgroup of $S_X$, and let $p$ be the largest prime divisor of $n$.  
		Assume that $|\Aut(R_2)|$ is coprime to $n$, where $R_2$ is a Sylow $2$-subgroup of $R$.
		
		Then for each regular subgroup $T$ isomorphic to $R$, either: 
		\begin{enumerate}
			\item\label{pabblo1} there exists $g\in\sg{R,T}$ such that the group $\sg{R,T^g}$ has a normal block system with blocks of size $p$; or
			\item\label{pabblo2} $n=12$, $R = \Z_3\rtimes\sg{y}$ with $o(y)=4$, and $\sg{R,T}$ is one of two transitive that roups contain nonconjugate regular subgroups isomorphic to $R$ and are of order $288$, and $1152$; or 
			\item\label{pabblo3} $n=24$, $R = \Z_3\rtimes\sg{y}$ with $o(y)=8$, and $\sg{R,T}$ is one of seventeen groups that contain nonconjugate regular subgroups isomorphic to $R$, with one of order $576$, two of order $2304$, five of order $9\,216$, four of order $36\,864$, two of order $147\,456$, two of order $589\,824$, and one of $2\,359\,296$.
		\end{enumerate} 
	\end{thrm}
	
	\begin{proof}
		We argue by
		induction on $|R|$. Let ${\mathcal  B}$ be a minimal  block system of $G=\langle R,T\rangle$ and let $k$ be the cardinality of its blocks.  If $k$ is divisible by an odd prime $q$, then by Proposition~\ref{300921a} there exists $g_1\in \sg{R,T}$ such that $\sg{R,T^{g_1}}$ has a normal block system ${\mathcal  C}$ with blocks of size $q$. If $q = p$, then~\eqref{pabblo1} is satisfied.  Otherwise, $q < p$.  Then $R^{\mathcal  C}\cong (T^{g_1})^{\mathcal  C}$, and so by induction either
		\begin{itemize}
			\item there exists $g_2\in\sg{R,T^{g_1}}$ such that $\sg{R^{\mathcal{C}},(T^{g_1g_2})^{\mathcal{C}}}$ has a normal block system ${\mathcal  D}/{\mathcal  C}$, where ${\mathcal  D}/{\mathcal  C}$ has blocks of size $p$ and ${\mathcal  D}$ is a block system of $\sg{T,T^{g_1g_2}}$ with blocks of size $qp$; or
			\item $|\mathcal{C}|\in \{12,24\}$.
		\end{itemize}
		In the second case, we contradict $q<p$, because $3$ is the largest prime dividing $12,24$.  In the first case, for every $D\in{\mathcal  D}$, a Sylow $p$-subgroup of $(\fix_{\sg{R,T^{g_1g_2}}}({\mathcal  D}))^D$ has 
		orbits of size $p$, and, therefore, has the same orbits as $(R_p)^D$.  By Lemma~\ref{2pbottom}, there is $g_3\in \sg{R,T^{g_1g_2}}$ such that $\sg{R,T^{g_1g_2g_3}}$ has a normal block system with blocks of size $p$. Thus~\eqref{pabblo1} follows with $g = g_1g_2g_3$.  So, we may assume that $k=2^d,d\leq e$.
		
		In what follows we assume that
		\begin{equation}
			\tag{$*$} \label{hypothesiseq}
			\forall g\in G, {\mathcal  B} \text{ is a minimal block system for } \sg{R,T^g}.
		\end{equation}
		
		Indeed, if ${\mathcal  B}$ is not minimal for some $\sg{R,T^g}$, then we can  replace $G$ by $\sg{R,T^g}$ and ${\mathcal  B}$ by a minimal  block system which refines ${\mathcal  B}$. This process may be continued until~\eqref{hypothesiseq} is satisfied. 

        \setcounter{case}{0}
		\begin{case}Assume that ${\mathcal  B}$ is not a normal block system of $R$.\end{case}

		In this case $R$ must have some non-normal subgroup, and hence $R\cong \Z_n\rtimes\sg{y}$ and $k=o(y)\in\{2,4,8\}$. 
		
		Since $\mathcal{B}$ is minimal and non-normal, its stabiliser $\fix_{G}(\mathcal{B})$ is trivial. 		If $k=4,8$, then $1\neq y^2\in Z(R)$ and, therefore, $y^2\in\fix_{G}(\mathcal{B})=\{1\}$. This contradiction implies $k=2$. 
		
		Now, the subgroups $(R_{2'})^{\mathcal  B}$ and $(T_{2'})^{\mathcal  B}$ are regular cyclic subgroups of $S_{{\mathcal  B}}$ of order $n$.  By \cite[Theorem 4.9]{Muzychuk1999} there exists $g\in \sg{R_{2'},T_{2'}}$ such that $\sg{R_{2'},T_{2'}^g}^{\mathcal  B}$ has a normal block system ${\mathcal  C}/{\mathcal  B}$ with blocks of size $p$. 
		By Sylow's  theorem there exists $g'$ such that $\langle R_p,(T_p^g)^{g'}\rangle$ is a $p$-group. Thus replacing $g$ by $gg'$, we can assume that $P=\langle R_p,T_p^g\rangle$ is a $p$-group.
		
		Let ${\mathcal D}$ be the set of orbits of $\fix_{\langle R, T^g\rangle}({\mathcal C})$.  
		Then ${\mathcal D}$ is a normal block system of $\langle R, T^g\rangle$, and it refines ${\mathcal C}$.  
		Since $R_p, T_p^g \le \fix_{\langle R, T^g\rangle}({\mathcal C})$, each block of ${\mathcal D}$ has size divisible by $p$.  
		Thus every block of ${\mathcal D}$ has size $\ell p$,  which is a multiple of $p$ and a divisor of $kp$; hence $\ell p$ is $p$ or $2p$.
		
		For each $C \in {\mathcal C}$, the groups $R_p$ and $T_p^g$ have $\ell$ orbits on $C$, each of size $p$.  
		Each such orbit is contained in an appropriate $P$-orbit on $C$.  
		Since $k=2 < p$,  $|P| < p^{2}$, which implies that all $P$-orbits on $C$ have size $p$ and therefore coincide with the corresponding orbits of $R_p$ and $T_p^g$.  
		In other words, $\mathrm{Orb}(R_p,C) = \mathrm{Orb}(T_p^g,C)$ for all $C \in {\mathcal C}$.
		By Lemma~\ref{2pbottom}, there exists $h \in \langle R, T^g\rangle$ such that  
		$\langle R, T^{gh}\rangle$ has a normal block system with block size $|R_p| = p$,  
		and part~\eqref{pabblo1} is satisfied.
		
		
		\begin{case}
		Assume that ${\mathcal  B}$ is a normal block system for $R$.
		\end{case}
		
		Then ${\mathcal B}$ is the set of orbits of some normal subgroup $R' \trianglelefteq R$, and $|R'| = k$.  
		We claim that ${\mathcal B}$ is normal for $T$ as well.  
		Indeed, for a given block $B \in {\mathcal B}$, its setwise stabiliser $T_{\{B\}}$ is a subgroup of order $k$.  
		Any group in ${\mathcal R}$ has the property that if a subgroup $X \le R \in {\mathcal R}$ is normal, then every subgroup of order $|X|$ is normal.  
		Applying this to $T_{\{B\}}$, we obtain $T' \trianglelefteq T$, and hence ${\mathcal B}$ is also normal for $T$.
		Therefore, $T^{\mathcal B} \cong T/T' \cong R/R' \cong R^{\mathcal B}$.
		
		By induction, either
		\begin{itemize}
			\item there exists $g_1 \in \langle R^{\mathcal B}, T^{\mathcal B} \rangle$ such that  
			$\langle R^{\mathcal B}, (T^{\mathcal B})^{g_1} \rangle$ has a normal block system with blocks of size $p$, or
			\item $|\mathcal{B}| \in \{12,24\}$.
		\end{itemize}
		In the second case, if $|\mathcal{B}| = 24$, we obtain a contradiction, since  
		$|R| = 2^{d} \cdot 24 = 2^{d+3} \cdot 3$,
		and no group in ${\mathcal R}$ has a Sylow $2$-subgroup of order greater than $8$ with a cyclic quotient of order $8$.  
		If $|\mathcal{B}| = 12$, then by Definition~\ref{190921a} we deduce that $R =\mathbb{Z}_3\rtimes\mathbb{Z}_8$, and part~\eqref{pabblo1} or~\eqref{pabblo3} follows from a computation in Magma.  
		Therefore, we may assume that the first case occurs.
		
		Let $g \in \langle R, T\rangle$ satisfy $g^{\mathcal B} = g_1$.  
		Then $
		\langle R^{\mathcal B}, (T^{\mathcal B})^{g_1} \rangle 
		= \langle R, T^{g} \rangle^{\mathcal B}$,
		which implies that $\langle R, T^{g} \rangle$ has a block system ${\mathcal C}$ with block size $kp$.
		
		If $p > k$, then, using the same argument as above and conjugating if necessary to ensure that  
		$R_p$ and $T_p$ lie inside a common $p$-subgroup, the hypotheses of Lemma~\ref{2pbottom} are satisfied.  
		Thus there exists $g' \in \langle R, T^g\rangle$ such that  
		$\langle R, T^{gg'} \rangle$ has a normal block system with blocks of size $|R_p| = p$, and part~\eqref{pabblo1} holds.  
		Therefore, it remains to consider the cases in which $k > p$.
		
		If $k = 4$, then $k > p$ implies that $n = p = 3$.  
		Since $R \in {\mathcal R}$ has a normal subgroup of order $k$, the possibilities for $R$ are  
		$\Z_{n} \times \Z_{4}$, $\Z_{n} \times Q_{8}$, $\Z_{n} \times \Z_{2}^{e}$ with $2 \le e \le 4$,  
		or $\Z_{n} \rtimes \langle y \rangle$ with $o(y) = 8$ and $y^{2} \in \mathsf{Z}(R)$.  
		As $|\Aut(R_{2})|$ is coprime to $n = 3$, all possibilities except  
		$\Z_{3} \times \Z_{4}$ and $\Z_{n} \rtimes \langle y \rangle$ with $o(y) = 8$ and  
		$y^{2} \in \mathsf{Z}(R)$ are eliminated.  
		The first case satisfies part~\eqref{pabblo1} by~\cite[Theorem~4.9]{Muzychuk1999}.  
		A computer computation shows that the second case satisfies either part~\eqref{pabblo1} or part~\eqref{pabblo3}.
		
		Assume now that $k > 4$, i.e.\ $k = 8,16$.  
		Since $R \in {\mathcal R}$ has a normal subgroup of order $k$, we have  
		$R_{2} \cong Q_{8}$ or $\Z_{2}^{e}$ with $e \in \{3,4\}$.  
		According to the classification of primitive groups of small degree \cite{DixonM1996},  
		the socle of $G_{\{B\}}^{B}$ is either elementary abelian, the alternating group $A_{B}$,  
		or $\PSL_{2}(7)$ in its natural action on $8$ points.  
		As $\mathrm{PGL}_{2}(7)$ contains neither regular elementary abelian subgroups of order $8$  
		nor regular subgroups isomorphic to $Q_{8}$, the last case cannot occur here.  
		Thus either $A_{B} \le G_{\{B\}}^{B}$ or $G_{\{B\}}^{B}$ is affine primitive.  
		In the first case, $\soc(G_{\{B\}}^{B}) \cong A_{k}$, and by Lemma~\ref{smaller-blocks}  
		we obtain a contradiction to~\eqref{hypothesiseq}.
		
		Assume now that $G_{\{B\}}^{B}$ is affine.  
		Suppose that $R_{2}$ is elementary abelian.  Denote by $E$ a subgroup of  $R_2$ satisfying $E = R_{\{B\}} \leq R_{2}$. Notice that $E$ is an elementary abelian $2$-group of order $k$.
		 
		For $B \in {\mathcal B}$, the group $\fix_{G}({\mathcal B})^{B}$ is isomorphic  
		to a subgroup of $E \rtimes \Aut(E)$, and hence, for  
		$C \in {\mathcal C}$ with $B \subseteq C$, the group  
		$\fix_{G}({\mathcal C})^{C}$ is isomorphic to a subgroup of  
		$
		(E \rtimes \Aut(E)) \wr \Z_{p}$.
		The assumption $\pi(|\Aut(R_{2})|) \cap \pi(n) = \emptyset$ implies that $\pi(\Aut(E))\cap\pi(n)=\emptyset$. Hence  
		$p^{2} \nmid |\fix_{G}({\mathcal C})^{C}|$.  
		Therefore, for each $C \in {\mathcal C}$, the subgroups $(R_{p})^{C}$  
		and $(T_{p})^{C}$ satisfy the hypotheses of Lemma~\ref{2pbottom}.  
		Applying Lemma~\ref{2pbottom}, we conclude that part~\eqref{pabblo1} is satisfied.

		Suppose that $R_{2}$ is not elementary abelian.  
		Then $R_{2} \cong Q_{8} \cong T_{2}$, and $k = 8$.  
		Let $r$ and $t$ be the unique central involutions of $R_{2}$ and $T_{2}$, respectively.  
		After conjugating $T_{2}$ if necessary, we may assume without loss of generality that  
		$\langle R_{2}, T_{2} \rangle$ is a $2$-group.  
		Let $z \in \langle R_{2}, T_{2} \rangle$ be an arbitrary central involution,  
		and let $B \in {\mathcal B}$.  
		Then $z^{B}$ centralizes both $R_{2}^{B}$ and $T_{2}^{B}$.  
		These are regular quaternion subgroups of $S_{B}$.  
		The only involution that centralizes a regular quaternion subgroup $Q_{8}$ in $S_{8}$  
		is the central involution of that $Q_{8}$.  
		Hence $z^{B} = r^{B}$, and analogously $z^{B} = t^{B}$.  
		Thus $r = t$, and consequently $\mathsf{Orb}(\langle r\rangle, X)$  
		is a block system of $\langle R, T\rangle$ with block size $2$.  
		But $\mathsf{Orb}(\langle r\rangle,X) \preceq {\mathcal B}$ and ${\mathcal B}$ is minimal.  
		Therefore $\mathsf{Orb}(\langle r\rangle,X) = {\mathcal B}$, which forces $k = 2$,  
		a contradiction.
	\end{proof}
	
	Not all permutation groups can be automorphism groups of particular types of combinatorial objects. In order for a permutation group $G$ to be the automorphism group of any colour digraph, it must be $2$-closed: that is, there cannot be any larger permutation group acting on the vertices that has the same orbits as $G$ on ordered pairs of vertices. Similarly, to be the automorphism group of a ternary relational structure, $G$ must be $3$-closed (no larger permutation group acting on the points has the same orbits on ordered triples). Observe that a permutation group is contained in its $3$-closure, which is contained in its $2$-closure, so being $2$-closed is a stronger property than being $3$-closed.
	
	\begin{remark}\label{examples of degree 12}
		In Theorem~\ref{041021a} part~\eqref{pabblo2}, the relevant transitive groups of degree $12$ are \texttt{TransitiveGroup}(12,$k$):
		\begin{itemize}
			\item $k=127$, of order $288$, which contains two conjugacy classes of regular subgroups isomorphic to $\Z_{3}\rtimes\langle y\rangle$;
			\item $k=204$, of order $1152$, which contains three such classes.
		\end{itemize}
		Neither of these groups is $3$-closed.
		
		In Theorem~\ref{041021a} part~\eqref{pabblo3}, the transitive groups of degree $24$ that arise are \texttt{TransitiveGroup}(24,$k$):
		\begin{itemize}
			\item $k=1479$, of order $576$, with two conjugacy classes of regular subgroups isomorphic to $\Z_{3}\rtimes\langle y\rangle$;
			\item $k=5042,5043$, both of order $2\,304$, each having two such classes;
			\item $k=9871,9878,9880,9883,9884$, all of order $9\,216$.  
			Among these, the first and third have two conjugacy classes of regular subgroups isomorphic to $\Z_{3}\rtimes\langle y\rangle$, while the others have three.
			\item $k=14269,14273,14510,14511$, all of order $36\,864$.  Among these, the first has three conjugacy classes of regular subgroups isomorphic to $\Z_{3}\rtimes\langle y\rangle$, while the others have four.
			\item $k=18102, 18314$, all of order $147\,456$, with four and five (respectively) conjugacy classes of regular subgroups isomorphic to $\Z_{3}\rtimes\langle y\rangle$.
			\item $k=20440, 20462$, all of order $589\, 824$, with four and six (respectively) conjugacy classes of regular subgroups isomorphic to $\Z_3\rtimes \langle y\rangle$.
			\item $k=22395$, of order $2\,359\,296$, which has six conjugacy classes of regular subgroups isomorphic to $\Z_3\rtimes \langle y\rangle$.
		\end{itemize}
		None of these groups is $3$-closed, except for \texttt{TransitiveGroup}(24,5043), which is not $2$-closed.
		
		No transitive group of degree 12 having more than one conjugacy class of regular dicyclic groups is $2$-closed, so if we move to the automorphism group of a binary relational structure, the regular dicyclic subgroups will always be conjugate. However, there are $2$-closed examples of degree 24 that have more than one conjugacy class of regular subgroups isomorphic to $\Z_3 \rtimes \sg{y}$ with $o(y)=8$. 
	\end{remark}
	
	From the statement of Theorem~\ref{041021a} it is clear that we can inductively find blocks of size $p$ where $p$ is the largest prime divisor of $|R|$ and then consider the action on these blocks, until and unless $|R|\in \{12, 24\}$. When $|R|=24$, by examining the proof of Theorem~\ref{041021a} we see that outcome (3) arises only when minimal blocks have size $2$ or $4$. When we move to consider the action on these blocks, there are either $6$ of them, in which case the induction proceeds again and we find $2$ blocks of size $3$, or there are $12$ of them. If at any point we are left with $12$ blocks, outcome (2) can arise if the minimal blocks have size $4$. Thus, the final sizes of blocks we find inductively can be any of the following: $4$ then $3$, $2$ then $4$ then $3$, $2$ then $3$ then $2$ then $2$, or $4$ then $3$ then $2$.
	With this analysis, Theorem~\ref{041021a} yields the following.
	\begin{cor}\label{111021a}
		Let $R\in {\mathcal  R}$, with  $|R|=2^e n$ where $0\le e\le 4$ and $n >1$ is odd and square-free, be a regular subgroup of $ S_{X}$. Assume that $|\mathrm{Aut}(R_2)|$ is coprime to $n$, where $R_2$ is a Sylow $2$-subgroup of $R$. 
		
		Then for each regular subgroup $T$ isomorphic to $R$, there exists $g\in\sg{R,T}$ such that the group $\sg{R,T^g}$  has a sequence of normal block systems
		$${\mathcal  B}_0\prec{\mathcal  B}_1\prec\cdots\prec{\mathcal  B}_{m - 1}\prec{\mathcal  B}_{m},$$
		where ${\mathcal  B}_0 $ consists of singleton sets, and ${\mathcal  B}_m$ has one block.  
		
		Furthermore, given $B_i\in\mathcal{B}_i$ for each $i$, except as noted below $\vert B_i\vert/\vert B_{i-1}|\ge \vert B_{i+1}\vert/\vert B_i\vert$ and both are prime for every $1\le i \le m-1$, and one of the following holds:
		\begin{itemize}
			\item $m=\Omega(2^en)$, and either:
			\begin{itemize}
				\item there are no exceptions, or
				\item $e=3$, $R/{\mathcal B}_{m-4} \cong \Z_3 \rtimes \sg{y}$ with $o(y)=8$, and $|B_{m-3}|/|B_{m-4}|=2$ while $|B_{m-2}|/|B_{m-3}|=3$; or
			\end{itemize}
			\item $m=\Omega(2^en)-1$, and either:
			\begin{itemize}
				\item $e=2$, $R/{\mathcal  B}_{m-2}$ is dicyclic, $|B_{m-1}|/|B_{m-2}|=4$, and $|B_m|/|B_{m-1}|=3$; or
				\item $e=3$, $R/{\mathcal B}_{m-3} \cong \Z_3 \rtimes \sg{y}$ with $o(y)=8$, and 
				\begin{itemize}
					\item $|B_{m-2}|/|B_{m-3}|=2$, $|B_{m-1}|/|B_{m-2}|=4$, and $|B_{m}|/|B_{m-1}|=3$, or 
					\item $|B_{m-2}|/|B_{m-3}|=4$, $|B_{m-1}|/|B_{m-2}|=3$, and $|B_{m}|/|B_{m-1}|=2$. 
				\end{itemize}
			\end{itemize}
		\end{itemize}
	\end{cor}
	
	When $m = \Omega(2^en)$, $\la R,T^g\ra$ is normally $m$-step imprimitive.
	
	A remark is also in order concerning the condition $\pi(|\Aut(R_2)|)\cap \pi(n)=\emptyset$.  
	For groups $R \in \mathcal{R}$, the possible Sylow $2$-subgroups are  
	$$R_2 \in \{1, \Z_2, \Z_2^2, \Z_2^3, \Z_2^4, \Z_4, \Z_8, Q_8\},$$  
	with $|\Aut(R_2)|$ equal to  
	$1, 1, 6, 168, 20160, 2, 4,$ and $24$, respectively.  
	Accordingly, the odd primes that must not divide $n$ are: none (for $1$ and $\Z_2$), $3$ (for $\Z_2^2$), $3$ or $7$ (for $\Z_2^3$), $3$, $5$, or $7$ (for $\Z_2^4$), none (for $\Z_4$ and $\Z_8$), and $3$ (for $Q_8$).  
	Theorem~4.2 in \cite{DobsonS2013} shows that several of these exclusions are genuinely necessary: in certain cases, groups whose orders are divisible by these primes are not CI-groups with respect to ternary relational structures.
	
	\section{Applications}
	
	The following result follows from Theorem \ref{041021a} and Theorem~\ref{dihedral-conditional}.  It generalizes \cite[Corollary 17]{Dobson2003}, which generalizes \cite[Theorem 22]{Dobson2002}, which in turn generalizes \cite[Theorem 4.4]{Babai1977}. The ``binary" piece follows because,  CI-groups with respect to ternary relational structures are CI-groups with respect to binary relational structures.
	
	\begin{thrm}\label{ternarymain}
		Let $n$ be an integer such that $\gcd(n,\varphi(n)) = 1$.  Then $\Z_n\rtimes\Z_2$ is a CI-group with respect to binary and ternary relational structures.
	\end{thrm}
	
	The following result is a combination of Theorem~\ref{ternarymain} and Corollary~\ref{previous-nonabelian}.
	
	\begin{cor}\label{dihedral cor}
		Let $n$ be odd.  Then $\Z_n\rtimes\Z_2$ is a CI-group with respect to ternary relational structures if and only if $\gcd(n,\varphi(n)) = 1$.
	\end{cor}

	The following result
	generalizes \cite[Corollary 29]{Dobson2003}.  The special case when $p$ is prime was proven for binary relational structures in \cite[Theorem 1.3]{LiLP2007}.
	
	\begin{thrm}\label{231221a}
		Let $n = km$ be an integer such that $\gcd(n,\varphi(n)) = 1$.  If $3\nmid m$, then $\Z_k\times(\Z_m\rtimes\Z_4)$ is a CI-group with respect to binary relational structures. If, in addition, no prime divisor $p$ of $n$ satisfies $p\equiv 1\ \pmod{4}$, then $\Z_k\times(\Z_m\rtimes\Z_4)$ is a CI-group with respect to ternary relational structures.
	\end{thrm}
	
	\begin{proof}
		Let $\Gamma$ be a binary Cayley relational structure of $\Z_k\times(\Z_m\rtimes\Z_4)$. Let $R\le\Aut(\Gamma)$ be regular and isomorphic to $\Z_k\times(\Z_m\rtimes\Z_4)$, and let $\delta\in S_{\Z_k\times(\Z_m\rtimes\Z_4)}$ be such that $R^\delta\in{\rm Aut}(\Gamma)$. By Corollary~\ref{111021a} there exists $g\in\sg{R,R^\delta}$ such that $\sg{R,R^{\delta g}}$ has a block system with four blocks of size $n$. By Theorem~\ref{dicyclic-conditional}, $\Gamma$ has the CI property.
		
		If every prime divisor $p$ of $n$ satisfies $p\equiv\,1 \ ({\rm mod}\ 4)$, then the same argument based on Corollary~\ref{111021a} and Theorem~\ref{dicyclic-conditional} yields that $\Z_k\times(\Z_m\rtimes\Z_4)$ is a CI-group with respect to ternary relational structures.
	\end{proof}
	
	In the next section  we prove the converse of the above result with respect to ternary relational structures: that is, we prove that if there is a prime divisor $p$ of $mk$ such that $p \equiv 1 \pmod{4}$, then $\Z_k \times (\Z_m \rtimes \Z_4)$ is not a CI-group with respect to ternary relational structures. This culminates in the characterization given in Corollary \ref{dicyclic cor}. 
	\section{On Inner Holomorphs that are 3-closed}\label{sec:holomorph}
	
	The \textbf{inner holomorph} of a finite group $G$ is the semidirect product $G \rtimes \inn(G)$.  
	It acts on $G$ via
	$x^{(g,h)} = h^{-1} x g h$, where $x,g,h \in G$.
	Under this action, the image of $G \rtimes \inn(G)$ in $S_G$ is the product $G_L \cdot G_R$, where  
	$G_L$ and $G_R$ denote the left and right regular representations of $G$, respectively.
	Another description of the inner holomorph is obtained as follows.  
	Let $G \times G$ act on $G$ by 
	$x^{(g,h)} = g^{-1} x h$.
	The stabilizer of the identity element is the diagonal subgroup $
	D = \{(g,g) : g \in G\}$,
	and the kernel of the action is 
	$\{(z,z) : z \in Z(G)\}.$
	The image of $G \times G$ in $S_G$ under this action is again $G_L \cdot G_R$;  
	the image of $G \times \{e\}$ is $G_L$, and the image of $\{e\} \times G$ is $G_R$.

	\begin{prop}\label{3closed} Let $F$ be a Frobenius group with abelian complement $H$. Assume that the natural action of $F$ on the right cosets of $H$ is $2$-closed. Then the inner holomorph of $F$ is $3$-closed provided that $|H|>2$.
	\end{prop}
	
	\begin{proof}
		Let $K$ be the Frobenius kernel of $F$ and let $G$ be the inner holomorph of $F$.  
		From the proof of \cite[Theorem 2.4]{OBrienPVV2022}, we have
		\[
		(G^{(3)})_e \leq (G_e)^{(2)},
		\]
		where $X^{(3)}$ and $X^{(2)}$ denote the $3$- and $2$-closures of a permutation group $X$, and where $e$ is the identity element of $F$.
		
		The group $G_e = \inn(F)$ is isomorphic to $F$ and acts on $F$ by conjugation.  
		Since $H$ is abelian, the set $F\setminus K$ splits into a disjoint union of $|H|-1$ orbits of $\inn(F)$, namely the nontrivial cosets $hK$ with $h\neq e$.  
		On the subset $K$, the group $\inn(F)$ has one orbit of size $1$ (namely $\{e\}$) and $\frac{|K|-1}{|H|}$ orbits of the form
		$k^H=\{h^{-1}kh\,:\, h\in H\},  k\in K$.
		
		The action of $\inn(F)$ on each coset $hK=Kh$ with $h\neq e$ is equivalent to the action of $F$ on the $H$-cosets, and is therefore $2$-closed.  
		Consequently,
		\[
		(\inn(F)^{(2)})^{hK} \leq (\inn(F)^{hK})^{(2)} = \inn(F)^{hK},
		\]
		and hence $(\inn(F)^{(2)})^{hK} = \inn(F)^{hK}$.  
		To show that $G$ is $3$-closed, it therefore suffices to prove that $(G^{(3)})_e$ acts faithfully on $hK$.
		
		Let $g \in (G^{(3)})_e \leq \inn(F)^{(2)}$ act trivially on $hK$.  
		Choose $a \in H\setminus\{e,h\}$.  
		The binary relation
		\[
		R_{h,a} = \{(h^k,a^k) : k \in K\}
		\]
		is a $2$-orbit of $\inn(F)$ and defines a bijection between the $\inn(F)$-orbits $hK$ and $aK$.  
		Since $R_{h,a}$ is $g$-invariant and $hK \subseteq \fix(g)$, we deduce that $aK$ is also fixed pointwise by $g$.  
		Thus any element that fixes $hK$ must act trivially on $F\setminus K$.
		
		Because left multiplication by $a$ sends $K$ to $aK$, the permutation
		$g_1 = a_L \, g \, a_L^{-1}$
		fixes $F\setminus aK$ pointwise.  
		Hence $g_1 \in (G^{(3)})_e \leq \inn(F)^{(2)}$ and $\fix(g_1) \supseteq hK$.  
		It follows that $g_1$ acts trivially on $F\setminus K$, and therefore $g_1$ is the identity.  
		Consequently $g$ is the identity as well.
	\end{proof}
	
	\begin{cor}\label{cor1} Let $p$ be a prime and let $\omega \in \Z_p^*$ be a primitive $n$th root of unity for some integer $n$ with $2 < n < p-1$. 
		Let $F = \Z_p \rtimes \langle \omega \rangle$ be the Frobenius group of order $np$. 
		Then the inner holomorph of $F$ is $3$-closed, and $F$ is not a CI-group with respect to ternary relational structures.
	\end{cor}
	
	\begin{proof} The group $F$, in its natural action on $p$ points, is $2$-closed since $n < p-1$.  
		By Proposition~\ref{3closed}, its inner holomorph is $3$-closed because $n > 2$.  
		Finally, the inner holomorph contains two non-conjugate regular subgroups isomorphic to $F$, namely $F_L$ and $F_R$.
	\end{proof}
	
	\begin{cor}\label{cor2}Let $p$ be a prime,  let $n$ be a divisor of $p-1$ with $2 < n < p-1$, and let $\alpha\colon \Z_n\rightarrow \Z_p^*$ be a non-injective group homomorphism. Assume that $|\Ker(\alpha)|$ is even whenever $n$ is even.  Then the semidirect product $\Z_p\rtimes_\alpha\Z_n$ is not a CI-group with respect to ternary relational structures.
	\end{cor}
	
	\begin{proof}Let $\omega$ be a primitive $n$th root of unity in $\Z_p$.  
		Without loss of generality we may assume that 
		$\alpha\colon \langle \omega \rangle \rightarrow \langle \omega \rangle$ 
		is an endomorphism of $\langle \omega \rangle$, that is, 
		$\alpha(\omega^i)=\omega^{ai}$ for some $a\in \Z_n$.  
		It follows from the assumptions that $a\notin \Z_n^*$, and that $a$ is even whenever $n$ is even.  
		These conditions ensure the existence of $b\in \Z_n^*$ such that $a-b$ is invertible in $\Z_n$.
		
		Consider the Frobenius group 
		$F=\Z_p\rtimes \Z_n = \{(x,\omega^i)\,:\,i\in\Z_n,\ x\in\Z_p\}$.  
		By Corollary~\ref{cor1}, the group $F_L\cdot F_R\cong F\times F$ is $3$-closed.  
		Write elements of $F\times F$ as $[(x,\omega^i),(y,\omega^j)]$ with $x,y\in\Z_p$ and $i,j\in\Z_n$.  
		In this notation we have:
		$$
		\begin{array}{rcl}
			F_L & = & \{[(x,\omega^i),(0,1)] \mid x\in \Z_p,\ i\in\Z_n\},\\ 
			F_R & = & \{[(0,1),(y,\omega^j)] \mid y\in \Z_p,\ j\in\Z_n\},\\
			(F_L\cdot F_R)_e = \mathsf{Inn}(F) 
			& = & 
			\{[(x,\omega^i),(x,\omega^i)] \mid x\in \Z_p,\ i\in\Z_n\}.
		\end{array}
		$$
		It is easy to check that the subsets
		\begin{align*}
			G_1 &= \{[(x,\omega^{ai}),(0,\omega^{bi})] : x\in\Z_p,\ i\in\Z_n\},\\
			G_2 &= \{[(0,\omega^{bi}),(y,\omega^{ai})] : y\in\Z_p,\ i\in\Z_n\}
		\end{align*}
		are two regular subgroups of $F_L\cdot F_R$, each isomorphic to $G$.  
		Regularity follows from the fact that 
		\[
		G_1 \cap \mathsf{Inn}(F)
		= \{[(0,\omega^{ai}),(0,\omega^{bi})] : ai = bi,\ i\in\Z_n\}
		\]
		has size $1$, and likewise for $G_2$.  
		Moreover, $G_1$ and $G_2$ are not conjugate in $F_L\cdot F_R$ because their Sylow $p$-subgroups $(G_i)_p$, $i=1,2$, are normal in $F_L\cdot F_R$ and are distinct.
	\end{proof}

	\begin{remark} If $\alpha$ is the trivial homomorphism, then the group $G$ in Corollary~\ref{cor2} is a direct product. If $\alpha$ is injective, then $\Z_p\rtimes_\alpha \Z_n$ is a Frobenius group of order $pn$.
	\end{remark}
	
	As an immediate consequence of Corollary~\ref{cor2} we obtain the following.

	\begin{cor}\label{cor3}
		Let $p > 5$ be a prime with $p \equiv 1 \pmod 4$. Then the groups $\Z_{4p}$, the Frobenius group $F_{4p}$ of order $4p$, and the dicyclic group of order $4p$ are not CI-groups with respect to ternary relational structures.
	\end{cor}
	
	\begin{example}\label{p is 5 sorted}{\rm
			There exists a unique permutation group of degree 20, up to permutation equivalence, that is 3-closed and that contains more than one conjugacy class of subgroups isomorphic to $H=\sg{x,y:x^5=y^4, x^y=x^{-1}}$.  This group has order $400$ and has exactly $2$ conjugacy classes of such subgroups, and is group $102$ in the Magma list of transitive groups of degree $20$.}
	\end{example}
	
	\begin{cor}\label{dicyclic cor}
		Let $n = km$ be odd with $3\nmid m$.  Then $\Z_k\times(\Z_m\rtimes\Z_4)$ is a CI-group with respect to ternary relational structures if and only if $\gcd(n,\varphi(n)) = 1$ and whenever $p$ divides $n$ is prime, $p$ satisfies $p\not\equiv 1 \pmod{4}$.
	\end{cor}
	
	\begin{proof} The sufficiency of the condition $p\not\equiv 1\ \pmod{4}$ follows from Theorem~\ref{231221a}. To prove the necessity we need to show that for each prime $p\equiv 1\ \pmod{4}$ the group $\Z_p\rtimes\Z_4$ is not CI with respect to ternary relational structures.
		
		If $\Z_4$ acts trivially on $\Z_p$, then $\Z_p\rtimes\Z_4\cong \Z_p\times\Z_4$ and the result follows from~\cite[Theorem 4.2]{DobsonS2013}. If the action is non-trivial, then the result follows from Corollary~\ref{cor3} when $p>5$ and Example \ref{p is 5 sorted} when $p=5$.
	\end{proof}
	
	\section*{Acknowledgements}
	The fourth author is funded by the European Union via the Next
	Generation EU (Mission 4 Component 1 CUP B53D23009410006, PRIN 2022, 2022PSTWLB, Group
	Theory and Applications). The second author's research is supported in part by an NSERC Discovery Research Grant, RGPIN-2024-04013.  This work is supported in part by the Slovenian Research and Innovation Agency (ARIS), research program P1-0285 and research projects J1-50000, J1-70035, J1-70047, and J1-70046.
	
    \providecommand{\bysame}{\leavevmode\hbox to3em{\hrulefill}\thinspace}
\providecommand{\MR}{\relax\ifhmode\unskip\space\fi MR }
\providecommand{\MRhref}[2]{%
  \href{http://www.ams.org/mathscinet-getitem?mr=#1}{#2}
}
\providecommand{\href}[2]{#2}

	
	
	
\end{document}